\documentclass{amsart}
\usepackage{mathrsfs,amsmath,amsfonts,enumerate,booktabs,siunitx,float,graphicx,color}

\newcommand{\threebar}[1]{{\left\vert\kern-0.25ex\left\vert\kern-0.25ex\left\vert #1 
    \right\vert\kern-0.25ex\right\vert\kern-0.25ex\right\vert}}
\newcommand{\A}{\mathsf{A}}
\newcommand{\B}{\mathsf{B}}
\newcommand{\N}{\mathbb{N}}

\newtheorem{question}{Question}
\newtheorem{lemma}{Lemma}
\newtheorem{proposition}[lemma]{Proposition}

\newtheorem{corollary}[lemma]{Corollary}

\newtheorem{theorem}{Theorem}

\newcommand{\R}{\mathbb{R}}

\begin{document}
\title[Barabanov norms which are not strictly convex]{An irreducible linear switching system whose unique Barabanov norm is not strictly convex}

\author{Ian D. Morris}

\maketitle

\begin{abstract}
We construct a marginally stable linear switching system in continuous time, in four dimensions and with three switching states, which is exponentially stable with respect to constant switching laws and which has a unique Barabanov norm, but such that the Barabanov norm fails to be strictly convex. This resolves a question of Y. Chitour, M. Gaye and P. Mason. \end{abstract}





\section{Introduction}

\subsection{Stability of linear switching systems}
For every $d \geq 1$ let $M_d(\R)$ denote the vector space of all $d \times d$ real matrices. If $\A \subset  M_d(\R)$ is a compact nonempty set, we will say that a trajectory of the \emph{linear switching system defined by $\A$} is any absolutely continuous function $x \colon [0,\infty) \to \R^d$ which solves a differential equation of the form
\begin{equation}\label{eq:that-equation}x'(t)=A(t)x(t)\end{equation}
in the sense of Carath\'eodory, for some Lebesgue measurable function $A \colon [0,\infty) \to\A$ which we refer to as a \emph{switching law}. This article is concerned with a question arising from the stability classification of linear switching systems. 

The stability regimes of a linear switching system are typically classified into the following categories (see e.g. \cite{ChMaSi22,Pr16,PrJu15,Su08}):  the linear switching system defined by $\A$ is called \emph{exponentially stable} if there exist $C, \kappa>0$ such that every trajectory $x(t)$ satisfies $\|x(t)\| \leq Ce^{-\kappa t}\|x(0)\|$ for all $t \geq 0$; \emph{exponentially unstable} if there exists a trajectory such that $\|x(t)\| \geq Ce^{\kappa t}\|x(0)\|>0$ for all $t \geq 0$, for some $C,\kappa>0$;  \emph{marginally stable} if neither of the two previous cases applies, but there exists $C>0$ such that every trajectory satisfies $\|x(t)\| \leq C\|x(0)\|$ for all $t \geq 0$; and \emph{marginally unstable} if none of the preceding conditions applies, in which case one may show that there necessarily exists an unbounded trajectory. If $\A$ is a singleton set with sole element $A$ then it is elementary that the four regimes may be distinguished via the spectrum of $A$ as follows: if $\lambda(A)$ denotes the maximum of the real parts of the eigenvalues of $A$ then the system is exponentially stable if $\lambda(A)<0$ and exponentially unstable if $\lambda(A)>0$;  if $\lambda(A)=0$ then it is marginally stable if every eigenvalue with zero real part is semisimple, and marginally unstable otherwise.

In the case where $\A$ is not a singleton set the situation is considerably more complicated, and the search for verifiable criteria with which to determine the stability of the linear switching system defined by a set $\A$ has given rise to a substantial body of research (see for example \cite{AgLi01,BaBoMa09,Bo02,FaMaCh09,GuShMa07,MaLa03,PyRa96,ShWiMaWuKi07}). By analogy with the case where $\A$ is a singleton set one may define the \emph{topological top Lyapunov exponent} of $\A$ to be the quantity
\[\Lambda(\A):=\lim_{t \to \infty} \sup_{x} \frac{1}{t}\log \left(\frac{\|x(t)\|}{\|x(0)\|}\right)\]
where the supremum is over all trajectories of the linear switching system defined by $\A$ such that $x(0)$ is nonzero. The existence of this limit is guaranteed by a subadditivity argument, and if $\A$ is a singleton set $\{A\}$ then one simply has $\lambda(A)=\Lambda(\A)$. It is not difficult to show that exponential stability occurs precisely when $\Lambda(\A)<0$ and exponential instability occurs precisely when $\Lambda(\A)>0$. One may also show that the quantity $\Lambda(\A)$ -- and indeed the classification of $\A$ among the four stability regimes -- is unaffected if $\A$ is replaced by its convex hull, and for this reason it is common to make the assumption that $\A$ is convex (see \cite{ShCaCu00}). On the other hand the accurate computation of the quantity $\Lambda(\A)$ remains an unresolved problem, with concise descriptions of $\Lambda(\A)$ being available only in very simple cases (see e.g. \cite{AgLi01,BaBoMa09,Bo02}). In this article the set $\A$ will usually be the convex hull of a finite set of matrices $\{A_1,\ldots,A_N\}$, and we will sometimes refer to the $N$ matrices $A_i$ as the \emph{switching states} of the linear switching system defined by $\A$.

\subsection{Extremal norms and most unstable switching laws}
A major source of difficulty in computing the value of $\Lambda(\A)$ is simply that the class of all possible switching laws is extremely large, and consequently the set of all trajectories which must be considered is, at least \emph{a priori}, also extremely large. In response to this problem a natural strategy is to try to reduce the breadth of the class of switching laws which needs to be considered, hopefully by identifying some small class of ``worst-case'' or ``most-unstable'' switching laws with the property that the system is (for example) exponentially stable if and only if every trajectory $x(t)$ of \eqref{eq:that-equation} in which $A(t)$ is a worst-case switching law converges to the origin. (This strategy is applied for example in \cite{Bo02,MaLa03,PyRa96,Ra96}, and also motivates works such as \cite{GuShMa07}.) A precise mathematical meaning may be given to the concept of ``most unstable switching law'' via the notions of \emph{extremal norm} and \emph{Barabanov norm}, as follows. A norm $\threebar{\cdot}$ on $\R^d$ is called \emph{extremal} for the (compact, nonempty) set $\A\subset M_d(\R)$ if the function $t\mapsto e^{-t\Lambda(\A)}\threebar{x(t)}$ is non-increasing along all trajectories $x$ of the linear switching system defined by $\A$. An extremal norm is called a \emph{Barabanov norm} for $\A$ if additionally for every $u \in \R^d$ there exists a trajectory $x$ of the linear switching system defined by $\A$ such that $x(0)=u$ and such that $e^{-t\Lambda(\A)}\threebar{x(t)}$ is constant with respect to $t \geq 0$. This allows one to define an \emph{extremal trajectory} for a given set $\A$ to be a trajectory $x$ such that $\threebar{x(t)} = e^{t\Lambda(\A)}\threebar{x(0)}>0$ for every $t \geq 0$ for some Barabanov norm $\threebar{\cdot}$ for $\A$, and a \emph{most-unstable switching law} to be a switching law $A \colon [0,\infty) \to \A$ such that $x'(t)=A(t)x(t)$ a.e. for some extremal trajectory $x$. If a given set $\A$  is replaced with the set $\hat{\A}:=\{A-\Lambda(\A)I\colon A \in \A\}$ then it is not difficult to see that $\Lambda(\hat\A)=0$, that the extremal norms and Barabanov norms of $\A$ are identical with those of $\hat{\A}$, and that $x \colon [0,\infty) \to \R^d$ is an extremal trajectory for $\hat\A$ if and only if $y(t):=e^{t\Lambda(\A)}x(t)$ is an extremal trajectory for $\A$. The study of extremal norms, Barabanov norms, extremal trajectories and most-unstable switching laws is therefore not substantially affected by the assumption that $\Lambda(\A)=0$ and we will frequently make this assumption in the sequel. Subject to this reduction it is clear that an extremal trajectory is precisely a trajectory which is confined to a level set of a Barabanov norm, and the problem of understanding most-unstable switching laws is thus inherently tied to that of understanding the geometry of Barabanov norms.

\subsection{Geometry of Barabanov norms} Let us say that a set $\A\subset M_d(\R)$ is \emph{reducible} if there exists a linear subspace $V \subset \R^d$ such that $AV \subseteq V$ for every $A \in \A$ and such that $0<\dim V <d$. Let us also say that $\A$ is \emph{irreducible} if it is not reducible. If $\A\subset  M_d(\R)$ is compact and irreducible then the existence of at least one Barabanov norm for $\A$ is guaranteed by a foundational result due to N.E. Barabanov \cite{Ba88b}. It is clear that if $\threebar{\cdot}$ is a Barabanov norm for $\A$ then so too is every positive scalar multiple of $\threebar{\cdot}$ and in this sense the Barabanov norm of a set $\A$ is never unique. However, we will say that $\A$ has a unique Barabanov norm up to scalar multiplication if the ratio of any two Barabanov norms for $\A$ is constant, and by mild abuse of language we will sometimes find it convenient to say that $\A$ has a unique Barabanov norm when this property holds for $\A$. Some sufficient conditions for the uniqueness of Barabanov norms in this sense have been given in \cite{ChGaMa15}, and some examples of non-uniqueness of Barabanov norms are also presented in that article.
  
In \cite{ChGaMa15}, Y. Chitour, M. Gaye and P. Mason considered the question of whether the unit ball of a Barabanov norm is necessarily strictly convex. Relative to this question they noted an example $\mathsf{M}\subset M_2(\R)$ with two switching states given by
\begin{equation}\label{eq:cgm}M_0:=\begin{pmatrix}0&0\\0&-1\end{pmatrix},\qquad M_1:=\begin{pmatrix}\alpha &3\\ -\frac{3}{5}& \frac{7}{10}\end{pmatrix}\end{equation}
where the real number $\alpha\simeq -0.88964\ldots$ is chosen in such a way that the curve $t \mapsto e^{tA_1}(-1,0)^T$ touches the vertical line $\{(1,y) \colon y\in \R\}$ tangentially for a unique positive $t$, and does not enter the open half-plane to the right of that line at any positive time $t$.
This system admits a unique Barabanov norm (see \cite{ChGaMa15}) and this norm is \emph{not} strictly convex. The unit ball of this norm is illustrated in Figure \ref{fi:one}. For this system there always exists a \emph{unique} extremal trajectory originating at each point in $\R^2\setminus \{0\}$. When the initial condition $x(0)$ lies in either of the two grey (shaded) sectors in Figure \ref{fi:one} this unique extremal trajectory has constant switching law $M_0$ and proceeds along a straight line towards the horizontal axis. When the initial condition instead lies in the white (unshaded) region the unique extremal trajectory follows a constant switching law corresponding to $M_1$ until it reaches one of the two grey (shaded) sectors. At this point the switching law changes to $M_0$ and the trajectory again converges to the horizontal axis along a straight vertical line.
\begin{figure}
\begin{center}
\includegraphics[scale=0.5]{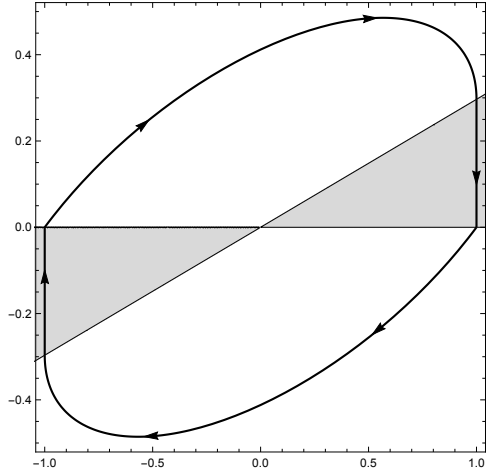}
\caption{The boundary of the unit ball of the Barabanov norm for the example \eqref{eq:cgm} defined by Chitour, Gaye and Mason. In the shaded region, the boundary of the unit ball is a straight vertical line.}\label{fi:one}
\end{center}
\end{figure}
Chitour, Gaye and Mason's example makes use of a zero eigenvalue in the matrix $M_0$ to confine certain trajectories to straight vertical lines, and this naturally raises the question of whether the existence of such a zero eigenvalue is a necessary condition for a system to admit a Barabanov norm which is not strictly convex. This question was posed as Open Problem 4 in the article \cite{ChGaMa15}. The purpose of this article is to answer this question negatively, as follows:
\begin{theorem}\label{th:main}
Define a real number $\lambda \in \R$ by
\[\lambda:=\Lambda\left(\left\{
\begin{pmatrix} 0&-2\\ \frac{1}{2}&0 \end{pmatrix}, \begin{pmatrix} 0 &-\frac{1}{2}\\ 2&0\end{pmatrix}\right\}\right)\]
and define four matrices by
\[A_0:=\begin{pmatrix}0&0\\0&-1\end{pmatrix},\qquad A_1:=\begin{pmatrix}-1&1\\ -1&-1\end{pmatrix},\]
\[B_0=\begin{pmatrix} -\lambda &-2\\ \frac{1}{2}&-\lambda \end{pmatrix},\qquad B_1=\begin{pmatrix} -\lambda &-\frac{1}{2}\\ 2&-\lambda\end{pmatrix}.\]
Now define three matrices $X_0, X_1, X_2 \in M_4(\R)$ by
\[X_0:=A_0 \otimes I + I\otimes B_0 =\begin{pmatrix} 
-\lambda&-2&0&0\\ 
\frac{1}{2}&-\lambda&0&0\\ 
0&0&-\lambda-1&-2\\ 
0&0&\frac{1}{2}&-\lambda-1 \end{pmatrix},\]
\[X_1:=A_0 \otimes I + I \otimes B_1=\begin{pmatrix} 
-\lambda&-\frac{1}{2}&0&0\\ 
2&-\lambda&0&0\\ 
0&0&-\lambda-1&-\frac{1}{2}\\ 
0&0&2&-\lambda-1 \end{pmatrix}, \]
\[X_2:=A_1 \otimes I=
\begin{pmatrix}
-1&0&1&0\\ 
0&-1&0&1\\ 
-1&0&-1&0\\
0&-1&0&-1 \end{pmatrix}
\]
and let $\mathsf{X}\subset M_4(\R)$ denote the convex hull of $\{X_0, X_1, X_2\}$. Then the linear switching system defined by $\mathsf{X}$ is irreducible and marginally stable, and every element of $\mathsf{X}$ is a Hurwitz matrix. There exists a unique Barbanov norm for $\mathsf{X}$ up to scalar multiplication, and this norm is not strictly convex.\end{theorem}
Here and elsewhere $A\otimes B$ denotes the Kronecker tensor product of the matrices $A$ and $B$, the properties of which are briefly recalled for the reader's benefit in \S\ref{se:two} below.

Intuitively, the mechanism for constructing this example is to ``expand'' the one-dimensional zero eigenspace of $M_0$ into a two-dimensional subspace which is invariant under the two matrices $X_0$ and $X_1$ in such a way that $\{X_0, X_1\}$ is marginally stable when restricted to this subspace. From a more precise viewpoint, the example $\mathsf{X}$ may be seen as a three-state subsystem of the six-state linear switching system whose states are
\begin{align*}A_0 \otimes I &+ I\otimes 0,&A_1 \otimes I &+ I\otimes 0,\\
A_0 \otimes I &+ I\otimes B_0, &A_1 \otimes I &+ I\otimes B_0,\\
A_0 \otimes I &+ I\otimes B_1, &A_1 \otimes I &+ I\otimes B_1. \end{align*}
This system admits trajectories which are given by the Kronecker tensor product of a trajectory of the system $\{A_0, A_1\}$ with a trajectory of the system $\{0, B_0, B_1\}$. In broad terms, our strategy will be to show that $\{A_0, A_1\}$ has a unique Barabanov norm, that this norm is not strictly convex, and that Barabanov norms of the above four-dimensional system are related to those of the two two-dimensional systems in a suitable way. This will allow us to show that the failure of strict convexity for the Barabanov norm of $\{A_0, A_1\}$ is inherited by the four-dimensional, six-state system defined above. Since the six-state system includes a matrix $A_0\otimes I$ which has zero eigenvalues it does not itself answer Chitour, Gaye and Mason's question, but it is possible to show that the three-state system $\{X_0, X_1, X_2\}$ is irreducible, has no Hurwitz matrices in its convex hull, and has the same Barabanov norms as the six-state system, which allows us to obtain the conclusion of Theorem \ref{th:main}. In practice this line of argument may be somewhat simplified by working directly with the three-state system in the first instance and omitting any reference to the six-state system described above, and this will be the approach which we will take in proving Theorem \ref{th:main}.

Similarly to the case for the one-dimensional eigenspace which is preserved by $M_0$ but not by $M_1$, the two-dimensional subspace preserved by $X_0$ and $X_1$ is invariant only for two of the three matrices $X_i$ and not for the system $\mathsf{X}$ as a whole. These two examples therefore have in common the feature that there exists a proper subset of the switching states which is marginally stable but preserves an invariant subspace. This suggests the following modification of Chitour, Gaye and Mason's open problem:
\begin{question}\label{qu:ackquack-motherfuckers}
Let $\A\subset M_d(\R)$ be compact and irreducible, and suppose that there does \emph{not} exist a compact nonempty proper subset $\B$ of $\A$ which is reducible and satisfies $\Lambda(\B)=\Lambda(\A)$. Is it the case that every Barabanov norm for $\A$ is strictly convex?
\end{question}

\subsection{Relationship with the discrete-time problem}

As an aside, we remark that for the corresponding problem of \emph{discrete-time} linear switching systems the study of the geometry of extremal norms is both highly developed and highly successful. In the discrete-time problem, instead of \eqref{eq:that-equation} one considers a trajectory to be a sequence $(x_n)_{n=1}^\infty$ of vectors in $\R^d$ such that $x_{n+1}=A_nx_n$ for all $n \geq 1$, for some sequence $(A_n) \in \A^\N$. Here the set $\A \subset M_d(\R)$ is compact and nonempty but is no longer assumed convex. Similarly to the definition of $\Lambda(\A)$ above one may consider the supremal rate of exponential growth of trajectories of the linear switching system defined by a set $\A$, which in this case coincides with the logarithm of the \emph{joint spectral radius},
\[\varrho(\A):=\lim_{n \to \infty} \max_{A_1,\ldots,A_n \in \A} \left\|A_n\cdots A_1\right\|^{\frac{1}{n}},\]
a quantity which was introduced in \cite{RoSt60} and which is now the subject of a very extensive literature which we do not attempt to encompass here.
The same four stability regimes may be defined for the discrete-time linear switching system, and extremal norms and Barabanov norms may be defined in a similar manner. In the series of articles \cite{GuPr13,GuPr16,GuZe08,GuZe09,GuZe15,Me20,Pr22} it was demonstrated that in ``typical'' cases there exists a Barabanov norm which is either piecewise affine or piecewise quadratic, and that using these norms the joint spectral radius can be computed explicitly by an algorithmic procedure. For discrete-time linear switching systems the classification of $\A$ as exponentially stable, exponentially unstable or otherwise may therefore be satisfactorily resolved  in typical cases by exploiting the simple geometry of the associated Barabanov norm. While the situation in the continuous-time case is far less developed, we suggest that the continued investigation of the geometry of Barabanov norms holds some hope of the possibilty of developing a similar framework for the computation of the quantity $\Lambda(\A)$ and consequently for the characterisation of the stability regimes of linear switching systems. 

\subsection{Structure of the article} In \S2 below we present some general facts concerning linear switching systems which will be useful throughout subsequent proofs. In the succeeding two sections we investigate the planar linear switching systems defined by the sets $\{A_0, A_1\}$,  $\{B_0, B_1\}$ and $\{0, B_0, B_1\}$, drawing on ideas from earlier works such as \cite{Bo02,ChGaMa15,Ra96}. In the final section of this article we combine the results of these analyses to prove Theorem \ref{th:main}.

\section{Preliminaries}\label{se:two}

\subsection{Notation} In this section we collect together some fundamental results which will assist our exposition. Henceforth all convex sets $\A\subset M_d(\R)$ will be the convex hull of a finite set $\{A_0,\ldots,A_{N}\}\subset M_d(\R)$. Since every element of $\A$ may be written as a convex combination of the elements $A_0,\ldots,A_{N}$, instead of representing switching laws directly as functions $A \colon [0,\infty) \to \A$ we will prefer to represent them in the form
\[A(t)=\left(1-\sum_{i=1}^{N}\alpha_i(t)\right)A_0 + \sum_{i=1}^{N} \alpha_i(t)A_i\]
where each $\alpha_i$ is a measurable function $[0,\infty) \to [0,1]$ and where $\sum_{i=1}^{N}\alpha_i(t)\leq1$ a.e. In this formulation \eqref{eq:that-equation} becomes
\begin{equation}\label{eq:nanas}x'(t)=\left(1-\sum_{i=1}^{N}\alpha_i(t)\right)A_0x(t)+  \sum_{i=1}^{N} \alpha_i(t)A_ix(t)\end{equation}
and we will subsequently use only this version of \eqref{eq:that-equation} when considering trajectories of linear switching systems.
Throughout the remainder of this article $\|\cdot\|$ will always denote either the Euclidean norm or the corresponding operator norm on matrices, according to context; the symbol $\threebar{\cdot}$ will be reserved for more general norms. 

\subsection{The topological top Lyapunov exponent}
The following key lemma will allow us to simplify the calculation of the quantity $\Lambda(\A)$ by approximating trajectories using matrix exponentials. 
\begin{lemma}\label{le:feuer-frei}
Let $A_0,\ldots,A_N \in M_d(\R)$, let $\mathsf{A}$ denote the convex hull of the set $\{A_0,\ldots,A_N\}$ and let $x \colon [0,\infty) \to \R^d$ be a trajectory of the linear switching system defined by $\mathsf{A}$. Then for every $T>0$ and $\varepsilon>0$ there exist $k \geq 1$, $t_1,\ldots,t_k>0$ and $i_1,\ldots,i_k \in \{0,\ldots,N\}$ such that $\sum_{j=1}^k t_j=T$ and $\|e^{t_k A_{i_k}}\cdots e^{t_1 A_{i_1}} x(0)- x(T)\|\leq \varepsilon$. 
\end{lemma}
\begin{proof}Suppose that $\alpha_1,\ldots,\alpha_N \colon [0,\infty) \to [0,1]$ are measurable functions such that $\sum_{i=1}^N \alpha_i \leq 1$ and such that $x$ solves \eqref{eq:nanas} for a.e. $t \geq 0$. Define $C:=T^{-1}+\max_i \|A_i\|$. Gr\"onwall's inequality implies that $\sup_{t \in [0,T]}\|x(t)\|\leq e^{CT}\|x(0)\|$ and it follows easily that $x$ is Lipschitz continuous on $[0,T]$ with Lipschitz constant bounded by $K:=Ce^{CT}\|x(0)\|$.  Fix an integer $k \geq 1$ large enough that $1/k \leq \varepsilon/4CKTe^{CT}$ and define $N$ step functions $\hat\alpha_1,\ldots,\hat\alpha_N \colon [0,\infty) \to \{0,1\}$ as follows. Define $\hat\alpha_i(t)=0$ for all $t \geq T$. For each $j=0,\ldots,k-1$ it is clear that we may partition the interval $[jT/k, (j+1)T/k)$ into $N+1$ disjoint half-open intervals with respective lengths $\int_{jT/k}^{(j+1)T/k} \alpha_1(s)ds$, \ldots, $\int_{jT/k}^{(j+1)T/k} \alpha_N(s)ds$ and $\int_{jT/k}^{(j+1)T/k} (1-\sum_{i=1}^N \alpha_i(s))ds$. (Here a half-open interval with zero length is of course the empty set.) We define $\hat\alpha_1(t)$ to be $1$ on the first of these intervals and zero on the remainder; $\hat\alpha_2(t)$ to be $1$ on the second interval and zero on the others; and so on, defining $\hat\alpha_N(t)$ to be $1$ on the $N^{\mathrm{th}}$ interval and zero on the remainder. In particular every $\hat\alpha_i$ is zero on the $(N+1)^{\mathrm{st}}$ interval. This construction clearly satisfies $0 \leq \sum_{i=1}^N \hat\alpha_i \leq 1$ and also satisfies
\begin{equation}\label{eq:bongly}\int_{\frac{jT}{k}}^{\frac{(j+1)T}{k}} \alpha_i(s)ds = \int_{\frac{jT}{k}}^{\frac{(j+1)T}{k}} \hat\alpha_i(s)ds \end{equation}
for every $i=1,\ldots,N$ and $j=0,\ldots,k-1$. 

Let $\hat{x} \colon [0,\infty) \to \R^d$ be the solution to the Carath\'eodory initial value problem
\[\hat{x}'(t)=\left(1-\sum_{i=1}^{N}\hat\alpha_i(t)\right)A_0\hat{x}(t)+  \sum_{i=1}^{N} \hat\alpha_i(t)A_i\hat{x}(t),\qquad \hat{x}(0)=x(0).\]
 Identical arguments show that $\hat{x}$ is also $K$-Lipschitz on $[0,T]$. By an intuitively clear (but tedious and notationally awkward) inductive argument, $\hat{x}(T)$ has the form
 \[\hat{x}(T)=e^{t_{k(N+1)} A_N}e^{t_{k(N+1)-1} A_{N-1}}\cdots e^{t_{(k-1)(N+1)+1}A_0} \cdots e^{t_1 A_0} x(0)\]
 for some real numbers $t_1,\ldots,t_{k(N+1)}\geq 0$ which correspond to the lengths of the intervals in the construction and which therefore sum to $T$. Clearly by removing those terms in which $t_j=0$ we may render this expression into the form given in the statement of the lemma, so to conclude the proof of the lemma it suffices to show that $\|x(T)-\hat{x}(T)\|\leq \varepsilon$. By Gr\"onwall's inequality this will follow if
 \begin{equation}\label{eq:gron}\|x(t)-\hat x(t)\| \leq \frac{\varepsilon}{e^{CT}} + C\int_0^t \|x(s)-\hat{x}(s)\|ds\end{equation}
for every $t \in [0,T)$, and this is what we now demonstrate.

For every $j=0,\ldots,k-1$ the expressions
\begin{equation}\label{eq:common-1} \int_{\frac{jT}{k}}^{\frac{(j+1)T}{k}}  \left(1- \sum_{i=1}^N\alpha_i(s)\right)A_0 x\left(\frac{jT}{k}\right)+  \sum_{i=1}^N\alpha_i(s)A_i x\left(\frac{jT}{k}\right)ds\end{equation}
and
\begin{equation}\label{eq:common-2} \int_{\frac{jT}{k}}^{\frac{(j+1)T}{k}} \left(1- \sum_{i=1}^N\hat\alpha_i(s)\right)A_0 x\left(\frac{jT}{k}\right)+  \sum_{i=1}^N\hat\alpha_i(s)A_i x\left(\frac{jT}{k}\right)ds\end{equation}
are identical by virtue of \eqref{eq:bongly}. By Lipschitz continuity and the choice of $k$ we have $\|x(s)-x(jT/k)\| <K/k\leq \varepsilon/4CTe^{CT}$ for every $s \in [jT/k, (j+1)T/k)$ and it follows that the integrals
\begin{equation}\label{eq:uncommon-1} \int_{\frac{jT}{k}}^{\frac{(j+1)T}{k}} \left(1- \sum_{i=1}^N\alpha_i(s)\right)A_0 x\left(s\right)+  \sum_{i=1}^N\alpha_i(s)A_i x\left(s\right)ds\end{equation}
and
\begin{equation}\label{eq:uncommon-2} \int_{\frac{jT}{k}}^{\frac{(j+1)T}{k}}  \left(1- \sum_{i=1}^N\hat\alpha_i(s)\right)A_0 x\left(s\right)+  \sum_{i=1}^N\hat\alpha_i(s)A_i x\left(s\right)ds\end{equation}
respectively differ from \eqref{eq:common-1} and \eqref{eq:common-2} by no more than $\varepsilon/4ke^{CT}$. Since the integrals \eqref{eq:common-1} and \eqref{eq:common-2} are equal, the difference in norm between \eqref{eq:uncommon-1} and \eqref{eq:uncommon-2}  is therefore not greater than  $\varepsilon/2ke^{CT}$. For every $j=0,\ldots,k$ the difference in norm between the integral
\begin{equation}\label{eq:deriv-1}\int_0^{\frac{jT}{k}}x'(s)ds=\int_0^{\frac{jT}{k}}\left(1- \sum_{i=1}^N\alpha_i(s)\right)A_0 x\left(s\right)+  \sum_{i=1}^N\alpha_i(s)A_i x\left(s\right)ds\end{equation}
and the integral
\begin{equation}\label{eq:deriv-2}\int_0^{\frac{jT}{k}} \left(1- \sum_{i=1}^N\hat\alpha_i(s)\right)A_0 x\left(s\right)+  \sum_{i=1}^N\hat\alpha_i(s)A_i x\left(s\right)ds\end{equation}
is thus bounded by $\varepsilon/2e^{CT}$. The difference between \eqref{eq:deriv-1} and the integral
\begin{equation}\label{eq:deriv-4}\int_0^{\frac{jT}{k}} \hat{x}'(s)ds= \int_0^{\frac{jT}{k}} \left(1- \sum_{i=1}^N\hat\alpha_i(s)\right)A_0 \hat x\left(s\right)+  \sum_{i=1}^N\hat\alpha_i(s)A_i \hat x\left(s\right)dt\end{equation}
is therefore bounded in norm by $\varepsilon/2e^{CT}$ plus the norm of the difference of \eqref{eq:deriv-2} and \eqref{eq:deriv-4}, hence is less than or equal to
\[\frac{\varepsilon}{2e^{CT}}+C\int_0^{\frac{jT}{k}} \left\|x(s)-\hat{x}(s)\right\|ds.\]
We have now shown that for every $j=0,\ldots,k-1$,
\begin{align*}\left\|x\left(\frac{jT}{k}\right) -\hat{x}\left(\frac{jT}{k}\right)\right\| &=\left\| \int_0^{\frac{jT}{k}} {x}'(s)ds -\int_0^{\frac{jT}{k}} \hat{x}'(s)ds\right\|\\
&\leq \frac{\varepsilon}{2e^{CT}}+C\int_0^{\frac{jT}{k}} \left\|x(s)-\hat{x}(s)\right\|ds.\end{align*}
Now given $t \in [0,T)$, choose $j$ such that $jT/k \leq t <(j+1)T/k$. Using the $K$-Lipschitz continuity of $x$ and $\hat{x}$ on $[0,T]$ we have $\|x(jT/k)-x(t)\| < K/k\leq  \varepsilon/4CTe^{CT}\leq \varepsilon/4e^{CT}$ and likewise $\|\hat{x}(jT/K)-\hat{x}(t)\|< \varepsilon/4e^{CT}$, so by the reverse triangle inequality
\[\left|\left\|x(t)-\hat{x}(t)\right\| - \left\|x\left(\frac{jT}{k}\right) -\hat{x}\left(\frac{jT}{k}\right)\right\|\right| \leq \frac{\varepsilon}{2e^{CT}}.\]
Consequently
\begin{align*}\left\|x(t)-\hat{x}(t)\right\| &\leq  \left\|x\left(\frac{jT}{k}\right) -\hat{x}\left(\frac{jT}{k}\right)\right\|+ \frac{\varepsilon}{2e^{CT}}\\
&\leq \frac{\varepsilon}{e^{CT}}+C\int_0^{\frac{jT}{k}} \left\|x(s)-\hat{x}(s)\right\|ds\\
&\leq \frac{\varepsilon}{e^{CT}}+C\int_0^t \left\|x(s)-\hat{x}(s)\right\|ds\end{align*}
and we have proved \eqref{eq:gron}. Since $t \in [0,T)$ was arbitrary the desired bound $\|\hat{x}(T)-x(T)\|\leq\varepsilon$ follows by Gr\"onwall's inequality and the proof of the lemma is complete.
\end{proof}
The following characterisation follows directly from Lemma \ref{le:feuer-frei} and the definition of $\Lambda(\A)$.
\begin{corollary}\label{co:eval}
Let $\mathsf{A}$ denote the convex hull of $\{A_0,\ldots,A_N\}\subset M_d(\R)$. Then
\[\Lambda(\A)=\lim_{T \to \infty} \sup_{k \geq 1} \sup_{\substack{t_1,\ldots,t_k>0\\ \sum_{i=1}^k t_i=T}}\max_{0\leq i_1,\ldots,i_k \leq N} \frac{1}{T}\log\left\|e^{t_k A_{i_k}}\cdots e^{t_1 A_{i_1}}\right\|.\]
\end{corollary}
Using Corollary \ref{co:eval} we derive the following tool for computing values of $\Lambda$:
\begin{proposition}\label{pr:calculo}
Let $A_0,\ldots,A_{N_1}, B_0,\ldots,B_{N_2} \in M_d(\R)$ and let $\A$ and $\B$ denote the convex hulls of $\{A_0,\ldots,A_{N_1}\}$ and $\{B_0,\ldots,B_{N_2}\}$ respectively. Then:
\begin{enumerate}[(i)]
\item\label{it:obv}
If $\A \subseteq \B$ then $\Lambda(\A) \leq \Lambda(\B)$.
\item\label{it:less-obv}
Let $\mathsf{C}$ denote the convex hull of $\{A_0,\ldots,A_{N_1},B_0,\ldots,B_{N_2}\}$. If every $A_i$ commutes with every $B_j$, then $\Lambda(\mathsf{C})=\max\{\Lambda(\A), \Lambda(\B)\}$.
\item\label{it:even-less-obv}
Let $\mathsf{D}$ denote the convex hull of $\{A_i + B_j \colon 0 \leq i \leq N_1,\text{ }0 \leq j \leq N_2\}$. If every $A_i$ commutes with every $B_j$, then $\Lambda(\mathsf{D})\leq\Lambda(\A)+ \Lambda(\B)$.
\end{enumerate}
\end{proposition}
\begin{proof}
To prove \eqref{it:obv} we simply note that every trajectory of the linear switching system defined by $\A$ is also a trajectory of the system defined by $\B$, and the inequality  $\Lambda(\A)  \leq \Lambda(\B)$ follows directly since the former quantity is a supremum over a smaller set than the latter. To prove \eqref{it:less-obv} some additional notation will be helpful: for every $T>0$ we define
\[\alpha(T):=\sup_{k \geq 1} \sup_{\substack{t_1,\ldots,t_k>0\\ \sum_{i=1}^k t_i=T}}\max_{0\leq i_1,\ldots,i_k \leq N_1} \log\left\|e^{t_k A_{i_k}}\cdots e^{t_1 A_{i_1}}\right\|,\]
\[\beta(T):=\sup_{\ell \geq 1} \sup_{\substack{\tau_1,\ldots,\tau_\ell>0\\ \sum_{j=1}^\ell \tau_j=T}}\max_{0\leq j_1,\ldots,j_\ell\leq N_2} \log\left\|e^{\tau_\ell B_{j_\ell}}\cdots e^{\tau_1 B_{j_1}}\right\|,\]
and let $\gamma(T)$ denote the corresponding quantity for $\mathsf{C}$. We then have
\[\Lambda(\A)=\lim_{T \to \infty} \frac{\alpha(T)}{T}, \quad \Lambda(\B)=\lim_{T \to \infty} \frac{\beta(T)}{T},\quad \Lambda(\A)=\lim_{T \to \infty} \frac{\gamma(T)}{T}.\]
Now, using commutativity any product of matrices of the form $e^{t_i A_i}$ or $e^{t_j B_j}$ may be permuted so as to have the matrices $e^{t_iA_i}$ collected together at the beginning of the product and the matrices $e^{t_jB_j}$ collected together at the end (but with the matrices $e^{t_iA_i}$ in the same order, among themselves, as they originally appeared in the product, and likewise with the matrices $e^{t_j B_j}$) without changing the value of the product. We may therefore write 
\[\gamma(T)= \sup_{k,\ell \geq 1} \sup_{\substack{t_1,\ldots,t_k>0\\ \tau_1,\ldots,\tau_\ell >0\\ \sum_{i=1}^k t_i+\sum_{j=1}^\ell \tau_j =T}}\max_{\substack{0\leq i_1,\ldots,i_k \leq N_1\\ 0 \leq j_1,\ldots,j_\ell \leq N_2}} \log\left\|e^{t_k A_{i_k}}\cdots e^{t_1 A_{i_1}}e^{\tau_\ell B_{j_\ell}}\cdots e^{\tau_1 B_{j_1}}\right\|\]
for every $T>0$, and using the elementary inequality 
\[\log\left\|e^{t_k A_{i_k}}\cdot\cdot e^{t_1 A_{i_1}}e^{\tau_\ell B_{j_\ell}}\cdot\cdot e^{\tau_1 B_{j_1}}\right\|\leq  \log\left\|e^{t_k A_{i_k}}\cdots e^{t_1 A_{i_1}}\right\|+ \log\left\|e^{\tau_\ell B_{j_\ell}}\cdots e^{\tau_1 B_{j_1}}\right\|\]
it follows easily that
\begin{align*}\gamma(T)&\leq \sup_{\substack{0 <T_1, T_2< T\\ T_1+T_2=T}} \alpha(T_1)+\beta(T_2)\\&\leq \sup_{\substack{0 <T_1, T_2< T\\ T_1+T_2=T}} \max\{\alpha(T_1), \beta(T_1)\}+\max\{\alpha(T_2), \beta(T_2)\}\end{align*}
for every $T>0$. It is now straightforward to deduce that
\[\Lambda(\mathsf{C}) =\lim_{T \to \infty} \frac{\gamma(T)}{T} \leq \lim_{T \to \infty} \frac{\max\{\alpha(T), \beta(T)\}}{T}=\max\{\Lambda(\A), \Lambda(\B)\}.\]
Since $\A\subseteq \mathsf{C}$ and $\B\subseteq \mathsf{C}$ the reverse inequality follows directly from \eqref{it:obv}, and this establishes \eqref{it:less-obv}. To prove \eqref{it:even-less-obv} we retain the notation for $\alpha(T)$ and $\beta(T)$, and also write
\[\delta(T):=\sup_{k \geq 1} \sup_{\substack{t_1,\ldots,t_k>0\\ \sum_{i=1}^k t_i=T}}\max_{\substack{0\leq i_1,\ldots,i_k \leq N_1\\ 0 \leq j_1,\ldots,j_k \leq N_2}} \log\left\|e^{t_k (A_{i_k}+B_{j_k})}\cdots e^{t_1 (A_{i_1}+B_{j_1})}\right\|\]
so that $\Lambda(\mathsf{D})=\lim_{T \to \infty} \delta(T)/T$. Since every $A_i$ commutes with every $B_j$, we have $e^{t(A_i+B_j)}=e^{tA_i}e^{tB_j}=e^{tB_j}e^{tA_i}$ for every $i,j$ and every $t \geq 0$. Consequently we may likewise permute the terms of the product so as to gather like terms together at either end, obtaining
\[\delta(T)=\sup_{k \geq 1} \sup_{\substack{t_1,\ldots,t_k>0\\ \sum_{i=1}^k t_i=T}}\max_{\substack{0\leq i_1,\ldots,i_k \leq N_1\\ 0 \leq j_1,\ldots,j_k \leq N_2}} \log\left\|e^{t_kA_{i_k}}\cdots e^{t_1A_1}e^{t_kB_{j_k}}\cdots e^{t_1 B_{j_1}}\right\|\]
and it follows easily that $\delta(T)\leq \alpha(T)+\beta(T)$ for every $T>0$. The result follows.
\end{proof}
\subsection{Extremal norms}
The following key result is proved in \cite{Ba88b}: 
\begin{proposition}\label{pr:baraban}
Let $\mathsf{A}\subset M_d(\R)$ be compact, convex and irreducible. Then there exists a Barabanov norm for $\mathsf{A}$.
\end{proposition}


The following straightforward result will give a simple condition for verifying that a given norm is extremal:
\begin{proposition}\label{pr:lyooo}
Let $\mathsf{A}$ denote the convex hull of $\{A_0,\ldots,A_N\}\subset M_d(\R)$ and let $f \colon \R^d \to \R$ be continuous. Then $t\mapsto f(x(t))$ is non-increasing for every trajectory $x$ of the linear switching system defined by $\A$ if and only if $t\mapsto f(e^{A_it}v)$ is non-increasing for every $i=0,\ldots,N$ and every $v \in \R^d$.\end{proposition}
\begin{proof}
Suppose that $t\mapsto f(e^{A_it}v)$ is non-increasing for every $i=0,\ldots,N$ and every $v \in \R^d$. It follows easily by induction on $k$ that for every $v \in \R^d$, $k\geq 1$, $t_1,\ldots,t_k>0$ and $0\leq i_1,\ldots,i_k \leq N$ we have $f(e^{t_k A_{i_k}} \cdots e^{t_1 A_{i_1}} v)\leq f(v)$. By Lemma \ref{le:feuer-frei} it follows directly that $f(x(t)) \leq f(x(0))$ for every $t \geq 0$, for every trajectory $x$ of the linear switching system defined by $\A$. If $t\mapsto x(t)$ is a trajectory of the linear switching system defined by $\A$ then so too is the function $t \mapsto x(t+\tau)$ for every $\tau \geq 0$, so we deduce that $f(x(t+\tau)) \leq f(x(\tau))$ for all $t, \tau \geq 0$ for every trajectory $x$ as required. On the other hand if  $t\mapsto f(x(t))$ is non-increasing for every trajectory $x$ then obviously $t\mapsto f(e^{tA_i}v)$ is non-increasing for every $i=0,\ldots,N$ and every $v \in \R^d$ as required, since each of the functions $t \mapsto e^{tA_i}v$ is a trajectory of the linear switching system defined by $\A$ with constant switching law $A_i$.\end{proof}

\subsection{Linear algebra}

We first recall some facts concerning the Kronecker product on vectors and matrices, which will be frequently appealed to without comment in the final section of this article. Proofs of all of the properties mentioned may be found in \cite[\S4.2]{HoJo94}. Given two rectangular real matrices $A$ and $B$ with respective dimensions $m_1 \times m_2$ and $n_1 \times n_2$, the Kronecker tensor product $A \otimes B$ is defined to be the $m_1n_1 \times m_2n_2$ matrix given by
\[A\otimes B=\begin{pmatrix}a_{11}B & \cdots &a_{1 m_2} B\\ 
\vdots &\ddots&\vdots\\
a_{m_11}B & \cdots &a_{m_1 m_2}B\end{pmatrix}.\]
The relations $\lambda (A\otimes B)=(\lambda A)\otimes B = A\otimes (\lambda B)$, where $\lambda \in \R$, are obvious. If $A_1$ and $A_2$ are rectangular matrices of the same dimensions then it is also clear that $(A_1+A_2) \otimes B=A_1\otimes B + A_2\otimes B$ for every rectangular matrix $B$, and likewise if $B_1$ and $B_2$ are rectangular matrices of the same dimensions then $A\otimes (B_1+B_2)=A\otimes B_1 + A\otimes B_2$ for every rectangular matrix $A$. The relation $\|A\otimes B\|=\|A\|\cdot \|B\|$ is satisfied for every pair of rectangular matrices $A$ and $B$, where $\|\cdot\|$ denotes the Euclidean operator norm on matrices. In particular if $u$ and $v$ are column vectors of respective dimensions $d_1$ and $d_2$ then $u\otimes v$ is a column vector of dimension $d_1d_2$ and with norm $\|u\otimes v\|=\|u\|\cdot\|v\|$. If $e_1,\ldots,e_{d_1}$ and $f_1,\ldots,f_{d_2}$ denote the standard bases of $\R^{d_1}$ and $\R^{d_2}$ respectively then it is easily seen that the vectors $e_i \otimes f_j$ make up the standard basis for $\R^{d_1d_2}$, and it is an easily exercise to deduce that every element of $\R^{d_1d_2}$ can be written as the sum of not more than $\min\{d_1, d_2\}$ vectors of the form $u \otimes v$. In particular  every element of $\R^4$ can be written in the form $u\otimes v + \hat{u}\otimes \hat{v}$ for some $u,v,\hat{u},\hat{v} \in \R^2$, a fact which will be applied in \S\ref{se:four} below.

If $A_1$ and $A_2$ are rectangular matrices of the same dimensions, $B_1$ and $B_2$ are rectangular matrices of the same dimensions, and the dimensions of these matrices are such that the products $A_1A_2$ and $B_1B_2$ are well-defined, then the relation $(A_1\otimes B_1)(A_2\otimes B_2)=(A_1A_2)\otimes (B_1B_2)$ is satisfied. If $A$ and $B$ are square matrices then the eigenvalues of $A\otimes B$ are precisely those numbers which may be obtained as the product of a eigenvalue of $A$ with an eigenvalue of $B$ . If $A$ and $B$ are square matrices then it is easily seen that the matrices $A\otimes I$ and $I\otimes B$ commute and that $e^{A\otimes I + I\otimes B}=(e^A \otimes I)(I\otimes e^B)=e^A\otimes e^B$, a fact which will be important in \S\ref{se:four}.

The following result can be obtained as a consequence of Schur's theorem on simultaneous triangularisation. For a proof we direct the reader to \cite[Theorem 2.4.8.1]{HoJo13}.
\begin{proposition}\label{pr:hojo}
Let $A, B \in M_d(\R)$ and suppose that $A$ and $B$ commute. Then every eigenvalue of $A+B$ is equal to the sum of an eigenvalue of $A$ and an eigenvalue of $B$.
\end{proposition}

\section{Analysis of the first two-dimensional system}

The following result modifies the earlier-mentioned example of Chitour, Gaye and Mason (specifically, \cite[Example 3]{ChGaMa15}) and presents the properties of the modified example at a level of detail which is suitable for later application.
\begin{theorem}\label{th:a-pepper}
Let $A_0, A_1 \in M_2(\R)$ be as defined in the statement of Theorem \ref{th:main}, and let $\mathsf{A}\subset M_2(\R)$ denote their convex hull. Then:
\begin{enumerate}[(i)]
\item\label{it:one}
The set $\mathsf{A}$ is irreducible and marginally stable, and every element of $\A$ is a Hurwitz matrix except for $A_0$.
\item\label{it:two}
The function $\threebar{\cdot}_\A \colon \R^2 \to \R$ defined by
\[\threebar{\begin{pmatrix}v_1 \\ v_2\end{pmatrix}}_\A = \sup_{\tau \geq 0} \left|e^{-\tau} v_1\cos \tau +e^{-\tau}v_2\sin\tau\right|\]
is a Barabanov norm for $\mathsf{A}$, and is the unique such norm up to scalar multiplication. This norm moreover satisfies
\begin{equation}\label{eq:stonks}\threebar{\begin{pmatrix}v_1 \\ v_2\end{pmatrix}}_\A=\threebar{\begin{pmatrix}v_1 \\ 0\end{pmatrix}}_\A  =|v_1|\end{equation}
when $|v_2|\leq |v_1|$.
\item\label{it:three}
Let $\alpha \colon[0,1]\to \R$ be Lebesgue measurable, and suppose that $x \colon [0,\infty) \to \R^2$ is absolutely continuous and satisfies 
\[x'(t)=(1-\alpha(t))A_0x(t)+\alpha(t)A_1x(t)\]
for a.e. $t \geq 0$. Then $x(t)$ converges as $t\to \infty$ to a point on the horizontal axis in $\R^2$, and if this limit is nonzero then $\int_0^\infty \alpha(t)dt<\infty$.
\end{enumerate}

\end{theorem}

\begin{proof}
We begin the proof by constructing and investigating the norm defined in \eqref{it:two}, and proceed to \eqref{it:one} and \eqref{it:three} only later. Define a norm on $\R^2$ by
\begin{align*}\threebar{\begin{pmatrix}v_1 \\ v_2\end{pmatrix}}_\A&:=\sup_{\tau \geq 0} \lim_{T\to\infty}\left\|e^{TA_0}e^{\tau A_1}\begin{pmatrix}v_1 \\ v_2\end{pmatrix}\right\|\\
&=\sup_{\tau \geq 0} \lim_{T\to\infty} \left\|\begin{pmatrix}1&0\\0&e^{-T}\end{pmatrix} \begin{pmatrix}e^{-\tau}\cos \tau & e^{-\tau} \sin \tau\\ -e^{-\tau}\sin \tau&e^{-\tau}\cos \tau\end{pmatrix}\begin{pmatrix}v_1\\v_2\end{pmatrix}\right\|\\
&= \sup_{\tau \geq 0} \left|e^{-\tau} v_1\cos \tau +e^{-\tau}v_2\sin\tau\right|\end{align*}
which coincides with the definition given in \eqref{it:two} above. It is clear that the above supremum is finite for all $v_1, v_2 \in \R$ and that $\threebar{\cdot}_\A$ satisfies the axioms of a norm. It is an easy calculus exercise to show that $\sup_{\tau \geq 0}e^{-\tau}(|\cos \tau |+|\sin \tau|)=1$, so if $|v_2|\leq |v_1|$ then
\[|v_1|\leq  \sup_{\tau \geq 0} \left|e^{-\tau} v_1\cos \tau +e^{-\tau}v_2\sin\tau\right| \leq |v_1| \left( \sup_{\tau \geq 0}e^{-\tau}(|\cos \tau |+|\sin \tau|)\right)=|v_1|\]
where the first inequality comes from the trivial observation that the supremum of $| v_1\cos \tau +e^{-\tau}v_2\sin\tau|$ over $\tau\geq 0$ is at least the value of that expression at $\tau=0$. 
In particular we have shown that
\[\threebar{\begin{pmatrix}v_1 \\ v_2\end{pmatrix}}_\A=\threebar{\begin{pmatrix}v_1 \\ 0\end{pmatrix}}_\A=|v_1|  \]
whenever $|v_2|\leq |v_1|$, as was claimed in \eqref{it:two}.

We next demonstrate that $\threebar{\cdot}_\A$ is non-increasing along every trajectory of the linear switching system defined by $\mathsf{A}$. By Proposition \ref{pr:lyooo} it is sufficient to show that $\threebar{e^{t A_i}v}_\A \leq \threebar{v}_\A$ for every $v \in \R^2$, $i=0,1$ and $t \geq 0$. The case $i=1$ is straightforward: if $t \geq0$ and $v \in \R^2$ then 
\[\threebar{e^{t A_1 }v}_\A=\sup_{\tau \geq 0} \lim_{T\to\infty}\left\|e^{TA_0}e^{\tau A_1} e^{tA_1}v\right\| = \sup_{\tau \geq t} \lim_{T\to\infty}\left\|e^{TA_0}e^{\tau A_1}v\right\|\leq \threebar{v}_\A.\]
To treat the case $i=0$ we argue as follows. Suppose that $v =(v_1,v_2)^T\in \R^2$ belongs to the closed unit ball of the norm $\threebar{\cdot}_\A$. We claim that necessarily $|v_1|\leq 1$. Indeed, if $v_1>1$ then for small enough $\varepsilon>0$ the vector $(1-\varepsilon)\cdot (1,0)^T +\varepsilon (v_1, v_2)^T$ has norm $(1-\varepsilon)+\varepsilon v_1>1$ by \eqref{eq:stonks}, but this vector also must belong to the unit ball of $\threebar{\cdot}_\A$ since it is a convex combination of two vectors which are elements of the closed unit ball. This is a contradiction and we conclude that $v_1 \leq 1$. If $v_1<-1$ then a similar contradiction arises by considering the norm of the vector $(1-\varepsilon)\cdot (-1,0)^T +\varepsilon (v_1, v_2)^T$ and we conclude that $|v_1|\leq 1$ as claimed. Since $|v_1|\leq 1$ it follows via \eqref{eq:stonks} that $(v_1,0)^T$ belongs to the closed unit ball. For every $t \geq 0$ the vector $e^{tA_0}v=(v_1, e^{-t}v_2)^T$ is a convex combination of the two vectors $(v_1,v_2)^T$ and $(v_1,0)^T$ both of which belong to the unit ball, and we conclude that $\threebar{e^{tA_0}v}_\A\leq 1=\threebar{v}_\A$ as required. By homogeneity the same result holds for every $v \in \R^2$. The hypotheses of Proposition \ref{pr:lyooo} are therefore satisfied and we conclude that $\threebar{x(t)}_\A$ is non-increasing for every trajectory $x(t)$ of the linear switching system defined by $\mathsf{A}$.

We may now prove \eqref{it:one}. A simple calculation shows that every $A \in \A$ other than $A_0$ has negative trace and positive determinant and is therefore Hurwitz. It is clear that $\mathsf{A}$ is irreducible: if every element of $\A$ preserves a nontrivial proper subspace of $\R^2$ then that subspace must be one-dimensional, but $A_1$ has no real eigenvectors and therefore cannot preserve a one-dimensional real vector space. Since $\threebar{\cdot}_\A$ is non-increasing along trajectories of the linear switching system defined by $\mathsf{A}$ it follows directly that  $\Lambda(\mathsf{A})\leq 0$. On the other hand since $e^{t A_0}u=u$ for all $t \geq 0$ whenever $u \in \R^2$ lies on the horizontal axis, the linear switching system defined by $\A$ admits constant trajectories and it follows that $\Lambda(\mathsf{A})\geq 0$. We conclude that $\Lambda(\mathsf{A})=0$ and that $\threebar{\cdot}_\A$ is an extremal norm for $\mathsf{A}$. In particular we have proved \eqref{it:one}. 

We next claim that $\threebar{\cdot}_\A$ is a Barabanov norm for $\A$. Let $v=(v_1, v_2)^T \in \R^2$ be nonzero, and observe that since clearly
\[\left|e^{-(\tau+\pi)} v_1\cos (\tau+\pi) +e^{-(\tau+\pi)}v_2\sin(\tau+\pi)\right| =e^{-\pi} \left|e^{-\tau} v_1\cos \tau +e^{-\tau}v_2\sin\tau\right| \]
for all $\tau\geq 0$, the supremum
\[\sup_{\tau \geq 0}   \left|e^{-\tau} v_1\cos \tau +e^{-\tau}v_2\sin\tau\right|=\sup_{\tau \geq 0}  \lim_{T\to\infty}\left\|e^{TA_1}e^{\tau A_0}\begin{pmatrix}v_1 \\ v_2\end{pmatrix}\right\|\]
is necessarily attained at some $\tau_0 \in [0,\pi)$. The function $x \colon [0,\infty) \to \R^2$ defined by $x(t)=e^{tA_1}v$ for $0 \leq t \leq \tau_0$ and $x(t)=e^{(t-\tau_0)A_0}e^{\tau_0 A_1}v$ for all $t >\tau_0$ is clearly a trajectory of the linear switching system defined by $\A$ and is easily verified to satisfy $\threebar{x(t)}_\A=\threebar{x(0)}_\A$ for all $t \geq 0$, so $\threebar{\cdot}_\A$ is a Barabanov norm for $\A$ as claimed.

We have now proved every clause of \eqref{it:one} and \eqref{it:two} except for the uniqueness of the Barabanov norm $\threebar{\cdot}_\A$, which we postpone until after proving \eqref{it:three}. Let $x \colon [0,\infty) \to \R^2$ satisfy \eqref{eq:that-equation} and note that since $\threebar{x(t)}_\A\leq \threebar{x(0)}_\A$ for all $t \geq 0$, there exists $C>0$ such that $\|x(t)\|\leq C$ for all $t \geq 0$. Let $x_1(t)$ and $x_2(t)$ denote the first and second co-ordinates of $x(t)$. For a.e. $t \geq 0$ we have
\[x_1'(t)=\alpha(t)(x_2(t)-x_1(t)),\]
\[x_2'(t)=-\alpha(t)x_1(t)-x_2(t).\]
It follows that for a.e. $t \geq 0$
\[2x_1'(t)x(t)+2x_2'(t)x_2(t) =-2\alpha(t)x_1(t)^2 -2x_2(t)^2\leq -2\alpha(t)(x_1(t)^2+x_2(t)^2)\]
and therefore
\[\|x(t)\|^2 \leq \exp \left(-2\int_0^t \alpha(s)ds\right)\|x(0)\|^2\]
for every $t \geq 0$. In particular, if $\int_0^\infty \alpha(t)dt$ is infinite then $\lim_{t \to \infty} x(t)=0$. Suppose instead that $\int_0^\infty \alpha(t)dt<\infty$. Since for a.e. $t \geq 0$ we have
\[x_2'(t)x_2(t) = -\alpha(t)x_1(t)x_2(t)-x_2(t)^2 \leq C^2\alpha(t)-x_2(t)^2\]
it follows that for all $t_2 \geq t_1 \geq 0$
\[e^{2t_2}x_2(t_2)^2 - e^{2t_1}x_2(t_1)^2 = \int_{t_1}^{t_2} 2e^{2s}x_2'(s)x_2(s) + 2e^{2s}x_2(s)^2 ds \leq 2C^2\int_{t_1}^{t_2}e^{2s}\alpha(s)ds\]
and therefore
\begin{align*}\left|x_2(t_2)\right|^2 &= e^{2(t_1-t_2)}x_2(t_1)^2 + 2C^2\int_{t_1}^{t_2}e^{2(s-t_2)}\alpha(s)ds\\
&\leq e^{2(t_1-t_2)}x_2(t_1)^2+2C^2\int_{t_1}^\infty \alpha(s)ds. \end{align*}
Taking the limit as $t_2 \to \infty$ yields
\[\limsup_{t_2 \to \infty} \left|x_2(t_2)\right|^2 \leq 2C^2 \int_{t_1}^\infty \alpha(s)ds\]
for all $t_1 \geq 0$, and since $\int_0^\infty \alpha(s)ds$ converges we conclude that $\lim_{t \to \infty} x_2(t)=0$. On the other hand clearly 
\[\int_0^\infty \left|x_1'(t)\right|dt = \int_0^\infty \alpha(t)\left|x_2(t)-x_1(t)\right| dt \leq 2C\int_0^\infty \alpha(t)dt<\infty\]
and this implies the existence of the limit $\lim_{t\to \infty}x_1(t)$. We have shown that in all cases $\lim_{t\to \infty}x(t)$ exists and lies on the horizontal axis, and if $\lim_{t\to \infty}x(t)$ is nonzero then $\int_0^\infty \alpha(s)ds<\infty$. This proves \eqref{it:three}.

To complete the proof we must establish the uniqueness of $\threebar{\cdot}_\A$ as a Barabanov norm. Suppose that $\threebar{\cdot}_1$ and $\threebar{\cdot}_2$ are both Barabanov norms for $\A$. By continuity and compactness there exists a vector on the Euclidean unit circle in $\R^2$ which attains the maximum value of the ratio $\threebar{\cdot }_1/\threebar{\cdot}_2$ on $\R^2\setminus \{0\}$, so using the Barabanov property of $\threebar{\cdot}_1$ let $x(t)$ be a trajectory such that $\threebar{x(t)}_1=\threebar{x(0)}_1$ for all $t \geq 0$ and such that $x(0)$ maximises  $\threebar{\cdot }_1/\threebar{\cdot}_2$. Since $\threebar{\cdot}_2$ is an extremal norm we have $\threebar{x(t)}_2\leq \threebar{x(0)}_2$ for all $t \geq 0$, so 
\[  \frac{\threebar{x(t)}_1}{\threebar{x(t)}_2} \leq  \frac{\threebar{x(0)}_1}{\threebar{x(0)}_2}\leq \frac{\threebar{x(0)}_1}{\threebar{x(t)}_2}  = \frac{\threebar{x(t)}_1}{\threebar{x(t)}_2}\]
for all $t \geq 0$, where in the first inequality we have used the fact that $x(0)$ maximises $\threebar{\cdot}_1/\threebar{\cdot}_2$. In particular we have $\threebar{x(t)}_1/\threebar{x(t)}_2 =\threebar{x(0)}_1/\threebar{x(0)}_2$ for every $t \geq 0$. By \eqref{it:two} the limit $\lim_{t\to \infty}x(t)$ exists and lies on the horizontal axis, and this limit is clearly nonzero since $\threebar{x(t)}_1=\threebar{x(0)}_1>0$ for every $t \geq 0$. We deduce that the ratio  $\threebar{\cdot}_1/\threebar{\cdot}_2$ attains its maximum on the horizontal axis. But this argument is symmetric with respect to the roles of $\threebar{\cdot}_1$ and $\threebar{\cdot}_2$, so the ratio $\threebar{\cdot}_2/\threebar{\cdot}_1$ \emph{also} attains its maximum on the horizontal axis. The ratio $\threebar{\cdot}_1/\threebar{\cdot}_2$ thus attains its maximum and its minimum at the same place, so its maximum and minimum are equal and it is perforce constant. We have proved the uniqueness of the Barabanov norm for $\A$ and the proof of the theorem is now complete.
\end{proof}

\section{Analysis of the second two-dimensional system}
The following result establishes some useful properties of the systems $\{B_0, B_1\}$ and $\{0, B_0, B_1\}$. It is similar in character to earlier works which also considered marginally stable linear switching systems whose trajectories rotate around the origin (such as \cite{PyRa96,Ra96} and parts of \cite{Bo02,MaLa03}) but the properties of the system which are of interest to us have mostly not been stated explicitly in earlier research. 
\begin{theorem}\label{th:consider-the-lilies}
Let $\lambda$, $B_0$ and $B_1$ be as in the statement of Theorem \ref{th:main}. Let $\mathsf{B}$ denote the convex hull of $\{B_0, B_1\}$ and let $\mathsf{B}_0$ denote the convex hull of $\{0, B_0, B_1\}$. Then: 
\begin{enumerate}[(i)]
\item\label{it:banan}
Both $\mathsf{B}$ and $\mathsf{B}_0$ are irreducible and marginally stable, and every element of $\mathsf{B}_0$ is Hurwitz except for the zero matrix.
\item\label{it:longcat}
There exists a norm $\threebar{\cdot}_{\mathsf{B}}$ on $\R^2$ which is a Barabanov norm for $\mathsf{B}$ and for $\mathsf{B}_0$. Up to scalar multiplication it is the unique extremal norm for $\mathsf{B}$ and is also the unique extremal norm for $\mathsf{B}_0$.
\item\label{it:orly}
Let $\alpha, \beta \colon [0,\infty) \to [0,1]$ be Lebesgue measurable functions such that $0 \leq \alpha(t)+\beta(t)\leq 1$ a.e, and let $y \colon [0,\infty) \to \R^2$ be an absolutely continuous function such that
\[y'(t)=(1-\alpha(t) - \beta(t))B_0y(t) + \beta(t)B_1y(t)\]
a.e. and such that $y(0)\neq 0$. Suppose that $\int_0^\infty (1-\alpha(t))dt$ diverges. Then for every nonzero $w \in \R^2$ there exists an increasing sequence of real numbers $(t_n)_{n=1}^\infty$ diverging to infinity such that $\|y(t_n)\|^{-1}y(t_n)= \|w\|^{-1}w$ for every $n \geq 1$.
\item\label{it:cheezburger}
For every $w \in \R^2$ there exists a Lebesgue measurable function $\beta \colon [0,\infty) \to [0,1]$ such that the solution $y \colon[0,\infty) \to \R^2$ to the initial value problem
\[y'(t)=(1-\beta(t))B_0y(t)+\beta(t)B_1y(t),\qquad y(0)=w\]
is periodic.
\end{enumerate}
\end{theorem}
\begin{proof}
A simple direct calculation using Corollary \ref{co:eval} shows that
\[\Lambda(\B)=\Lambda\left(\left\{
\begin{pmatrix} 0&-2\\ \frac{1}{2}&0 \end{pmatrix}, \begin{pmatrix} 0 &-\frac{1}{2}\\ 2&0\end{pmatrix}\right\}\right)-\lambda\]
which is zero by definition. Consequently
\[0=\Lambda(\B) \leq \Lambda(\B_0) \leq \max\left\{\Lambda(\B), \Lambda(\{0\})\right\}=0\]
where we have used parts \eqref{it:obv} and \eqref{it:less-obv} of Proposition \ref{pr:calculo}. Since $B_0$ has non-real eigenvalues it cannot preserve a one-dimensional subspace of $\R^2$ and therefore every subset of $M_2(\R)$ which includes $B_0$ is irreducible. In particular both $\B_0$ and $\B$ are irreducible. It follows from Proposition \ref{pr:baraban} that there exists a Barabanov norm $\threebar{\cdot}_\B$ for $\B$, and since every trajectory of $\mathsf{B}$ is a trajectory of $\mathsf{B}_0$ it follows easily that $\threebar{\cdot}_\B$ is also a Barabanov norm for $\B_0$. The existence of a Barabanov norm for $\B$ and $\B_0$ together with the fact that $\Lambda(\B_0)=\Lambda(\B)=0$ implies that $\B_0$ and $\B$ are marginally stable. Since by a simple calculation
\[\left\|\left(\exp\left(\frac{\pi}{2} \cdot \begin{pmatrix} 0&-2\\ \frac{1}{2}&0 \end{pmatrix}\right)\exp\left(\frac{\pi}{2} \cdot \begin{pmatrix} 0&-\frac{1}{2}\\ 2&0 \end{pmatrix}\right)\right)^n \right\|=4^n\]
for every $n \geq 1$, we have
\[\lambda=\Lambda\left(\left\{
\begin{pmatrix} 0&-2\\ \frac{1}{2}&0 \end{pmatrix}, \begin{pmatrix} 0 &-\frac{1}{2}\\ 2&0\end{pmatrix}\right\}\right)\geq\frac{\log 4}{\pi}\]
by Corollary \ref{co:eval} and in particular $\lambda$ is strictly positive. Clearly every element of $\mathsf{B}$ has trace $-2\lambda<0$ and a simple calculation shows that every element of $\B$ also has positive determinant. In combination these properties imply that every eigenvalue of every element of $\mathsf{B}$ has negative real part, which is to say that every element of $\B$ is Hurwitz. Since every nonzero element of $\B_0$ is a positive scalar multiple of an element of $\B$ it follows that every nonzero element of $\B_0$ is Hurwitz as required. We have proved \eqref{it:banan} and the existence clause of \eqref{it:longcat}.

In order to derive the uniqueness clause of \eqref{it:longcat} we must first prove \eqref{it:orly} and \eqref{it:cheezburger}. Suppose that $\alpha, \beta$ and $y$ are as in \eqref{it:orly} and that $\int_0^\infty (1-\alpha(t))dt$ is divergent, and let $w \in \R^2$ be an arbitrary nonzero vector. Define $r(\phi):=(\cos \phi, \sin \phi)^T \in \R^2$ for every $\phi \in \R$. Write $y(t)$ in polar co-ordinates as $y(t)=n(t)r(\theta(t))$ for all $t \geq 0$, where $n(t):=\|y(t)\|$ for all $t \geq 0$ and where $\theta \colon [0,\infty) \to \R$ is continuous. It is clear that $n$ and $\theta$ are absolutely continuous, so for a.e. $t \geq 0$ we have
\[y'(t)=n'(t)r(\theta(t)) + \theta'(t)n(t)r\left(\theta(t)+\frac{\pi}{2}\right)\]
and also
\begin{align*}y'(t)&=(1-\alpha(t) - \beta(t))B_0 + \beta(t)B_1)y(t)\\
&=n(t) ((
1-\alpha(t) - \beta(t))B_0 + \beta(t)B_1)r(\theta(t)).\end{align*}
Hence for a.e. $t \geq 0$,
\begin{align*}\theta'(t) &= \frac{1}{n(t)} \left\langle y'(t), r\left(\theta(t)+\frac{\pi}{2}\right)\right\rangle\\
&=\left\langle (1-\alpha(t) - \beta(t))B_0 + \beta(t)B_1)r(\theta(t)), r\left(\theta(t)+\frac{\pi}{2}\right)\right\rangle \\
&=\left\langle \begin{pmatrix} -\lambda & -2+\frac{3}{2} \beta(t)+2\alpha(t)\\\frac{1}{2}+\frac{3}{2}\beta(t)-\frac{1}{2}\alpha(t)&-\lambda\end{pmatrix}\begin{pmatrix} \cos \theta(t)\\ \sin \theta(t)\end{pmatrix} ,\begin{pmatrix} -\sin \theta(t)\\ \cos \theta(t)\end{pmatrix}\right\rangle\\
 &=\left(2-\frac{3}{2}\beta(t)-2\alpha(t)\right) \sin^2\theta(t) +  \left(\frac{1}{2}+\frac{3}{2}\beta(t)-\frac{1}{2}\alpha(t)\right)\cos^2\theta(t)\\
&=\frac{1-\alpha(t)}{2} + \frac{3}{2}\left(1-\alpha(t) - \beta(t)\right)\sin^2\theta(t)+\frac{3}{2}\beta(t)\cos^2\theta(t)\\
&\geq\frac{1-\alpha(t)}{2}\geq 0.\end{align*}
In particular $\theta$ is non-decreasing with respect to $t$. If $\theta(t)$ increased to a finite limit as $t \to \infty$ then the integral $\int_0^\infty \theta'(t)dt$ would be finite, implying the finiteness of the integral $\int_0^\infty (1-\alpha(t))dt$ and contradicting the hypotheses of the theorem. We conclude that instead $\lim_{t\to \infty} \theta(t)=+\infty$, and it follows that there exists an increasing sequence of positive real numbers $(t_n)_{n=1}^\infty$ diverging to infinity such that $r(\theta(t_n))=\|w\|^{-1}w$  for every $n \geq 1$. Clearly this implies $\|y(t_n)\|^{-1}y(t_n)=\|w\|^{-1}w$  for every $n \geq 1$ as required to prove \eqref{it:orly}.

We now prove \eqref{it:cheezburger}. If $w \in \R^2$ is the zero vector then every trajectory $y$ such that $y(t)=0$ is constant, so we suppose that $w \in \R^2$ is nonzero. Since $\threebar{\cdot}_\B$ is a Barabanov norm for $\B$ there exist a measurable function $\hat\beta \colon [0,\infty) \to \R$ and absolutely continuous function $\hat y \colon [0,\infty) \to \R^2$ such that 
\[\hat y'(t)=(1-\hat\beta(t))B_0\hat y(t)+\hat\beta(t)B_1\hat y(t),\qquad\hat y(0)=w\]
for a.e. $t \geq 0$ and such that additionally $\threebar{\hat y(t)}_\B=\threebar{\hat y(0)}_\B$ for all $t \geq 0$. It follows from \eqref{it:orly} that there exists $T>0$ such that $\|\hat y(T)\|^{-1}\hat y(T)=\|w\|^{-1}w=\|\hat y(0)\|^{-1}\hat  y(0)$. Consequently
\[\frac{\threebar{\hat y(T)}_\B}{\|\hat y(T)\|} =\frac{\threebar{\hat y(0)}_\B}{\|\hat y(0)\|}\]
and since $\threebar{\hat y(T)}_\B=\threebar{\hat y(0)}_\B$ this implies $\|\hat y(0)\|=\|\hat y(T)\|$. In particular we have $\hat y(T)=\hat y(0)$. Now define $\beta \colon [0,\infty) \to [0,1]$ to be periodic with period $T$ and such that $\beta(t)=\hat\beta(t)$ for all $t \in [0,T]$, and similarly let $y \colon [0,\infty) \to \R^2$ be periodic with period $T$ and such that $y(t)=\hat{y}(t)$ for all $t \in [0,T]$. It is clear that $y$ is absolutely continuous and periodic and solves an initial value problem of the type required for it to be a trajectory of the linear switching system defined by $\B$. This proves \eqref{it:cheezburger}.

We now prove that $\threebar{\cdot}_\B$ is the unique extremal norm for $\B$ up to scalar multiplication, which implies that it is also the unique extremal norm for $\B_0$. Suppose that $\threebar{\cdot}$ is also an extremal norm for $\B$. By \eqref{it:cheezburger} we may choose a periodic trajectory $y$ of the linear switching system defined by $\B$, where $y(0)$ is an arbitrary nonzero vector. Since $\threebar{\cdot}$ and $\threebar{\cdot}_\B$ are extremal norms for $\B$, the functions  $t \mapsto \threebar{y(t)}$ and $t \mapsto \threebar{y(t)}_\B$ are non-increasing; but by the periodicity of $y$, these two functions are also periodic. They are therefore constant, so the ratio $\threebar{y(t)}_\B/\threebar{y(t)}$ is constant with respect to $t \geq 0$. Now if $w \in \R^2$ is any nonzero vector, by \eqref{it:orly} there exists $t_0\geq0$ such that $\|y(t_0)\|^{-1}y(t_0)=\|w\|^{-1}w$. In particular 
\[\frac{\threebar{w}_\B}{\threebar{w}} =\frac{\threebar{\|w\|^{-1}w}_\B}{\threebar{\|w\|^{-1}w}} =\frac{\threebar{\|y(t_0)\|^{-1}y(t_0)}_\B}{\threebar{\|y(t_0)\|^{-1}y(t_0)}}=\frac{\threebar{y(t_0)}_\B}{\threebar{y(t_0)}}=\frac{\threebar{y(0)}_\B}{\threebar{y(0)}}\]
and since the nonzero vector $w \in \R^2$ was arbitrary, the ratio of $\threebar{\cdot}_\B$ to $\threebar{\cdot}$ is constant on $\R^2\setminus\{0\}$. Thus $\threebar{\cdot}_\B$ is a scalar multiple of $\threebar{\cdot}$ as required. The proof of the theorem is complete.
\end{proof}

\section{Proof of Theorem \ref{th:main}}\label{se:four}

\subsection{Irreducibility} We first prove that $\mathsf{X}$ is irreducible. Let $\mathscr{X}\subseteq M_4(\R)$ denote the algebra generated by $\mathsf{X}$, i.e. the linear span of the set of all finite nonempty products of elements of $\mathsf{X}$. If $V$ is a vector subspace of $\R^4$ which is preserved by every element of $\mathsf{X}$ then it is clear that it is also preserved by every element of $\mathscr{X}$, so to show that $\mathsf{X}$ is irreducible it is sufficient to show that $\mathscr{X}$ is irreducible. For each $i,j$ in the range $1 \leq i,j \leq 4$ let $E_{ij}$ denote the $4\times 4$ matrix with $1$ in position $(i,j)$ and all other entries zero. We will show that all of the matrices $E_{ij}$ are elements of $\mathscr{X}$, from which it follows that $\mathscr{X}=M_4(\R)$ and hence that $\mathsf{X}$ is irreducible.

The matrix
\[Z_1:=\frac{2}{3}\left(X_1-X_0\right)=\begin{pmatrix} 
0&1&0&0\\ 
1&0&0&0\\ 
0&0&0&1\\ 
0&0&1&0 \end{pmatrix}\]
is an element of $\mathscr{X}$, and therefore $\mathscr{X}$ also contains the identity matrix $I=Z_1^2$. It consequently contains the matrices
\[Z_2:=X_0+\lambda I=\begin{pmatrix} 
0&-2&0&0\\ 
\frac{1}{2}&0&0&0\\ 
0&0&-1&-2\\ 
0&0&\frac{1}{2}&-1 \end{pmatrix}, \qquad Z_3:=X_1+\lambda I=\begin{pmatrix} 
0&-\frac{1}{2}&0&0\\ 
2&0&0&0\\ 
0&0&-1&-\frac{1}{2}\\ 
0&0&2&-1 \end{pmatrix}\]
and their products
\[Z_4:=Z_2Z_3=\begin{pmatrix} 
-4&0&0&0\\ 
0&-\frac{1}{4}&0&0\\ 
0&0&-3&\frac{5}{2}\\ 
0&0&-\frac{5}{2}&\frac{3}{4} \end{pmatrix}, \qquad Z_5:=Z_3Z_2=\begin{pmatrix} 
-\frac{1}{4}&0&0&0\\ 
0&-4&0&0\\ 
0&0&\frac{3}{4}&\frac{5}{2}\\ 
0&0&-\frac{5}{2}&-3\end{pmatrix}.\]
The matrices
\[\begin{pmatrix} \frac{3}{4}&\frac{5}{2}\\ -\frac{5}{2}&-3\end{pmatrix}, \qquad \begin{pmatrix} -3&\frac{5}{2}\\ -\frac{5}{2}&\frac{3}{4}\end{pmatrix}\]
both have determinant $4$ and trace $-\frac{9}{4}$, so each has a conjugate pair of non-real eigenvalues with absolute value $2$. It follows that $-4$ is the unique eigenvalue of $Z_4$ with maximal modulus and that the same holds for $Z_5$. Consequently 
\[E_{11}=\lim_{n \to \infty} \left(-\frac{1}{4} Z_4\right)^n, \qquad E_{22}=\lim_{n \to \infty} \left(-\frac{1}{4} Z_5\right)^n\]
and the two matrices $E_{11}, E_{22}$ are elements of $\mathscr{X}$.  Since
\[Z_6:=Z_5-Z_4= \begin{pmatrix} 
\frac{15}{4}&0&0&0\\ 
0&-\frac{15}{4}&0&0\\ 
0&0&\frac{15}{4}&0\\ 
0&0&0&-\frac{15}{4}\end{pmatrix}\]
is an element of $\mathscr{X}$ it follows that so too is the matrix
\[Z_7:=\frac{4}{15} Z_6-E_{11}+E_{22} = \begin{pmatrix} 
0&0&0&0\\ 
0&0&0&0\\ 
0&0&1&0\\ 
0&0&0&-1\end{pmatrix}.\]
Consequently $E_{33}=\frac{1}{2}(Z_7^2+Z_7)\in \mathscr{X}$ and $E_{44}=\frac{1}{2}(Z_7^2-Z_7) \in \mathscr{X}$. We conclude that $\mathscr{X}$ contains the four matrices $E_{ii}$ for $i=1,2,3,4$. Now since
\[Z_8:= X_2+Z_1X_2=\begin{pmatrix}
-1&-1&1&1\\ 
-1&-1&1&1\\ 
-1&-1&-1&-1\\
-1&-1&-1&-1 \end{pmatrix}\]
is an element of $\mathscr{X}$, we note that for every $i,j=1,2,3,4$ we have
\[E_{ii}Z_8E_{jj} \in \{E_{ij}, -E_{ij}\}\]
and therefore $E_{ij} \in \mathscr{X}$ for every $i,j=1,2,3,4$. We deduce that $\mathscr{X}=M_4(\R)$ as claimed and we have proved that $\mathsf{X}$ is irreducible.

\subsection{Stability properties} We next show that $\mathsf{X}$ is marginally stable and that all of its elements are Hurwitz matrices. For every $k \geq 1$, $i_1,\ldots,i_k \in \{0,1\}$ and $t_1,\ldots,t_k>0$ we have 
\[\left\|e^{t_k X_{i_k}}\cdots e^{t_1 X_{i_1}}\right\| = \left\|e^{(t_1+\cdots+ t_k) A_0} \otimes \left(e^{t_k B_{i_k}}\cdots e^{t_1 B_{i_1}}\right)\right\| =\left\|e^{t_k B_{i_k}}\cdots e^{t_1 B_{i_1}}\right\|\]
where we have used the fact that $\|e^{tA_0}\|=1$ for every $t \geq 0$. It follows directly using Corollary \ref{co:eval} that $\Lambda(\mathsf{X}) \geq \Lambda(\B)$. On the other hand, since clearly
\[\mathsf{X} \subset \left\{A \otimes I + I \otimes B \colon A \in \mathsf{A}\text{ and }B \in \mathsf{B}_0\right\}\]
and since matrices of the form $A\otimes I$ and $I\otimes B$ commute, it follows using Proposition  \ref{pr:calculo}\eqref{it:even-less-obv} that $\Lambda(\mathsf{X}) \leq \Lambda(\mathsf{A})+\Lambda(\mathsf{B}_0)$. Combining these inequalities with Theorem \ref{th:a-pepper}\eqref{it:one} and Theorem \ref{th:consider-the-lilies}\eqref{it:banan} shows that $\Lambda(\mathsf{X})=0$. Since $\mathsf{X}$ is irreducible it admits a Barabanov norm by Proposition \ref{pr:baraban} and this implies that it is marginally stable.

Now let us show that every $X \in \mathsf{X}$ is Hurwitz. If $X\in \mathsf{X}$ is a convex combination of $X_0$ and $X_1$ only, then it has the form $A_0\otimes I + I \otimes B$ for some  $B \in \mathsf{B}$. By Theorem \ref{th:consider-the-lilies}\eqref{it:banan} the matrix $B$ is Hurwitz, and it is clear by inspection that every eigenvalue of $A_0$ has non-positive real part. Every eigenvalue of $A_0 \otimes I$ is an eigenvalue of $A_0$ and every eigenvalue of $I\otimes B$ is an eigenvalue of $B$, and the two matrices $A_0 \otimes I$ and $I\otimes B$ commute, so it follows by Proposition \ref{pr:hojo} that every eigenvalue of $X$ is the sum of a number with non-positive real part and a number with negative real part. Hence in this case $X$ is Hurwitz. If on the other hand $X$ is a convex combination of $X_0$, $X_1$ and $X_2$ with nonzero contribution from $X_2$, then it may be written in the form $A\otimes I + I \otimes B$ where $A \in \A$ is not equal to $A_0$ and where $B\in \B_0$. By Theorem \ref{th:a-pepper}\eqref{it:one} the matrix $A$ is Hurwitz, and by Theorem \ref{th:consider-the-lilies}\eqref{it:banan} the matrix $B$ is either Hurwitz or zero. It follows that $A\otimes I$ is Hurwitz and that $I\otimes B$ is either Hurwitz or zero, and we may again apply Proposition \ref{pr:hojo} to deduce that $X$ is Hurwitz as required. We have proved that every element of $\mathsf{X}$ is Hurwitz.

\subsection{Uniqueness of the Barabanov norm} Since $\mathsf{X}$ is irreducible it admits at least one Barabanov norm; we now wish to show that this norm is unique up to scalar multiplication. Our strategy will be as follows. Let $e_1, e_2 \in \R^2$ denote the standard basis vectors. Given any two Barabanov norms $\threebar{\cdot}_1$ and $\threebar{\cdot}_2$ for $\mathsf{X}$, we will show that the ratio $\threebar{\cdot}_1/\threebar{\cdot}_2$ attains its maximum at the vector $e_1 \otimes e_1 \in \R^4$. Since this result is equally applicable when the roles of $\threebar{\cdot}_1$ and $\threebar{\cdot}_2$ are interchanged, it follows that $\threebar{\cdot}_2/\threebar{\cdot}_1$ must attain its maximum at the same place; but then the maximum and minimum of the ratio $\threebar{\cdot}_1/\threebar{\cdot}_2$ are attained at the same location, which implies that the ratio of the two norms is constant as desired.
Let us therefore fix two Barabanov norms $\threebar{\cdot}_1$ and $\threebar{\cdot}_2$ for $\mathsf{X}$. We proceed in two stages: we first show that  $\threebar{\cdot}_1/\threebar{\cdot}_2$ is maximised at a vector of the form $e_1 \otimes v$, and then subsequently deduce that it is maximised at $e_1 \otimes e_1$. 

By homogeneity, every value taken by $\threebar{\cdot}_1/\threebar{\cdot}_2$ at a nonzero vector is also taken at a Euclidean unit vector, so by compactness this ratio attains its maximum at some nonzero vector in $\R^4$. Every element of $\R^4$ may be written in the form $u\otimes v + \hat{u}\otimes \hat{v}$ for some $u, v, \hat{u}, \hat{v}\in \R^2$, so choose such vectors $u, v, \hat{u}, \hat{v}$ with $u$ and $v$ nonzero such that $\threebar{\cdot}_1/\threebar{\cdot}_2$ is maximised at $u\otimes v + \hat{u}\otimes \hat{v}$. Since $\threebar{\cdot}_1$ is a Barabanov norm for $\mathsf{X}$ there exists a trajectory $z \colon [0,\infty) \to \R^4$ of the switched system defined by $\mathsf{X}$ such that $z(0)=u\otimes v + \hat{u}\otimes \hat{v}$ and $\threebar{z(t)}_1=\threebar{z(0)}_1$ for all $t \geq 0$. Let $\alpha, \beta \colon [0,\infty) \to [0,1]$ be measurable functions such that $0 \leq \alpha + \beta \leq 1$ and such that 
\[z'(t)= (1-\alpha(t) - \beta(t))X_0z(t) + \beta(t)X_1z(t)+\alpha(t) X_2 z(t)\]
for a.e. $t \geq 0$. Now let $x, y, \hat{x}, \hat{y} \colon [0,\infty) \to \R^2$ solve the initial value problems
\[x'(t) = (1-\alpha(t))A_0x(t)+\alpha(t)A_1x(t),\qquad x(0)=u,\]
\[\hat{x}'(t) = (1-\alpha(t))A_0\hat{x}(t)+\alpha(t)A_1\hat{x}(t),\qquad \hat{x}(0)=\hat{u},\]
\[y'(t) =(1-\alpha(t) - \beta(t))B_0y(t) + \beta(t)B_1y(t),\qquad y(0)=v,\]
\[\hat{y}'(t) =(1-\alpha(t) - \beta(t))B_0\hat{y}(t) + \beta(t)B_1\hat{y}(t),\qquad \hat{y}(0)=\hat{v}.\]
Since
\[z(0)=x(0)\otimes y(0)+ \hat{x}(0)\otimes \hat{y}(0)\]
and for a.e. $t \geq 0$
\[z'(t)=x'(t)\otimes y(t)+x(t)\otimes y'(t) + \hat{x}'(t)\otimes \hat{y}(t)+\hat{x}(y)\otimes \hat{y}'(t)\]
it follows that $z(t)$ and $x(t)\otimes y(t) + \hat{x}(t)\otimes \hat{y}(t)$ solve the same initial value problem and are consequently identical. Moreover $x$ and $\hat{x}$ are trajectories of the linear switched system defined by $\mathsf{A}$, and $y$ and $\hat{y}$ are trajectories of the linear switching system defined by $\B_0$. By Theorem \ref{th:a-pepper}\eqref{it:three} both $x(t)$ and $\hat{x}(t)$ converge as $t \to \infty$ to vectors on the horizontal axis in $\R^2$, so we have $\lim_{t \to \infty} x(t)= \gamma e_1$ and $\lim_{t \to \infty} x(t)= \hat\gamma e_1$ for some real numbers $\gamma, \hat\gamma$, say. Since $\B_0$ is marginally stable the trajectories $y$ and $\hat{y}$ are bounded, so we may choose an increasing sequence $(t_n)_{n=1}^\infty$ which diverges to infinity such that the limits $w:=\lim_{n \to \infty} y(t_n)$ and $\hat{w}:=\lim_{n \to \infty} \hat{y}(t_n)$ exist. In particular we have
\begin{align*}\lim_{n \to \infty} z(t_n)=\lim_{n \to \infty} x(t_n)\otimes y(t_n) + \hat{x}(t_n)\otimes \hat{y}(t_n)&= (\gamma e_1) \otimes  w + (\hat\gamma e_1)\otimes \hat{w}\\&=e_1 \otimes (\gamma w+\hat \gamma \hat w).\end{align*}
By definition $\threebar{z(t_n)}_1=\threebar{z(0)}_1$ for all $n \geq 1$, so $\threebar{e_1 \otimes (\gamma w+\hat \gamma \hat w)}_1=\threebar{z(0)}_1>0$ and the limit vector $e_1 \otimes (\gamma w+\hat \gamma \hat w)$ is nonzero. Since $\threebar{\cdot}_2$ is an extremal norm for $\mathsf{X}$ we have $\threebar{z(t_n)}_2\leq \threebar{z(0)}_2$ for all $n \geq 1$, so for all $n \geq 1$
\[\frac{\threebar{z(0)}_1}{\threebar{z(0)}_2} = \frac{\threebar{z(t_n)}_1}{\threebar{z(0)}_2} \leq  \frac{\threebar{z(t_n)}_1}{\threebar{z(t_n)}_2} \leq \frac{\threebar{z(0)}_1}{\threebar{z(0)}_2} \]
where the final inequality uses the fact that the ratio $\threebar{\cdot}_1/\threebar{\cdot}_2$ is maximised at $z(0)$. Taking the limit as $n \to \infty$ yields
\[ \frac{\threebar{z(0)}_1}{\threebar{z(0)}_2} = \frac{\threebar{e_1 \otimes (\gamma w+\hat \gamma \hat w)}_1}{\threebar{e_1 \otimes (\gamma w+\hat \gamma \hat w)}_2}\]
and we have shown that $\threebar{\cdot}_1/\threebar{\cdot}_2$ is maximised at a vector of the form claimed.

Now let $e_1 \otimes v$ be a vector at which $\threebar{\cdot}_1/\threebar{\cdot}_2$ is maximised, and let us show that $\threebar{\cdot}_1/\threebar{\cdot}_2$ is also maximised at $e_1 \otimes e_1$. Similarly to before, let $z \colon [0,\infty) \to \R^4$ be a trajectory of the switched system defined by $\mathsf{X}$ such that $z(0)=e_1 \otimes v$ and $\threebar{z(t)}_1=\threebar{z(0)}_1$ for all $t \geq 0$. Let $\alpha, \beta \colon [0,\infty) \to [0,1]$ be measurable functions such that $0 \leq \alpha + \beta \leq 1$ and such that 
\[z'(t)= (1-\alpha(t) - \beta(t))X_0z(t) + \beta(t)X_1z(t)+\alpha(t) X_2 z(t)\]
for a.e. $t \geq 0$, and let $x, y \colon [0,\infty) \to \R^2$ solve the initial value problems
\[x'(t) = (1-\alpha(t))A_0x(t)+\alpha(t)A_1x(t),\qquad x(0)=e_1,\]
\[y'(t) =(1-\alpha(t) - \beta(t))B_0y(t) + \beta(t)B_1y(t),\qquad y(0)=v.\]
Once more we have $z(t)=x(t)\otimes y(t)$ for all $t \geq 0$ since both functions solve the same initial value problem. Since $\threebar{x(t)\otimes y(t)}_1=\threebar{z(0)}_1>0$ for all $t \geq 0$ we have $\liminf_{t \to \infty} \|x(t)\otimes y(t)\|>0$, and since $y$ is a trajectory of the marginally stable linear switching system defined by $\B_0$, we have $\limsup_{t \to \infty} \|y(t)\|<\infty$. It follows that
\[\liminf_{t \to \infty} \|x(t)\| = \liminf_{t \to \infty} \frac{\|x(t) \otimes y(t)\|}{\|y(t)\|} >0\]
and in particular $x(t)$ does not converge to zero in the limit as $t \to \infty$. It follows by Theorem \ref{th:a-pepper} that $\lim_{t \to \infty} x(t)=\gamma_1 \cdot e_1$ for some nonzero real number $\gamma_1$ and that the integral $\int_0^\infty \alpha(t)dt$ is convergent. The integral $\int_0^\infty (1-\alpha(t))dt$ is therefore divergent, so by Theorem \ref{th:consider-the-lilies} there exists an increasing sequence $(t_n)_{n=1}^\infty$ which diverges to infinity such that $\|y(t_n)\|^{-1}y(t_n)=e_1$ for every $n \geq 1$. By passing to a subsequence we may assume without loss of generality that $\lim_{n \to \infty} y(t_n)$ exists, and we write this limit as $\gamma_2 \cdot e_1$. We therefore have $\lim_{n \to \infty} z(t_n) = \lim_{n \to \infty} x(t_n)\otimes y(t_n)=(\gamma_1 e_1)\otimes (\gamma_2 e_1)=\gamma (e_1 \otimes e_1)$, say. Since $\threebar{z(t_n)}_1=\threebar{z(0)}_1>0$ for all $n \geq 1$ it follows that $\gamma$ is nonzero. For every $n \geq 1$ we have
\[\frac{\threebar{z(0)}_1}{\threebar{z(0)}_2} = \frac{\threebar{z(t_n)}_1}{\threebar{z(0)}_2} \leq  \frac{\threebar{z(t_n)}_1}{\threebar{z(t_n)}_2} \leq \frac{\threebar{z(0)}_1}{\threebar{z(0)}_2} \]
using the fact that $\threebar{z(t_n)}_1=\threebar{z(0)}_1$ for all $n \geq 1$, the fact that $\threebar{z(t_n)}_2 \leq \threebar{z(0)}_2$ for all $n \geq 1$, and the fact that $\threebar{\cdot}_1/\threebar{\cdot}_2$  is maximised at $z(0)$. Taking the limit as $n \to \infty$ yields 
\[ \frac{\threebar{z(0)}_1}{\threebar{z(0)}_2} = \frac{\threebar{\gamma (e_1 \otimes e_1)}_1}{\threebar{\gamma (e_1 \otimes e_1)}_2}=\frac{\threebar{e_1 \otimes e_1}_1}{\threebar{e_1 \otimes e_1}_2} \]
and we have shown that $\threebar{\cdot}_1/\threebar{\cdot}_2$ is maximised at the vector $e_1 \otimes e_1$. 

Combining the above steps, we have shown that if $\threebar{\cdot}_1$ and $\threebar{\cdot }_2$ are any two Barabanov norms for $\mathsf{X}$ then the ratio $\threebar{\cdot}_1/\threebar{\cdot}_2$ is maximised at $e_1 \otimes e_1$. Since this applies equally with the roles of  $\threebar{\cdot}_1$ and $\threebar{\cdot }_2$ interchanged, the ratio $\threebar{\cdot}_2/\threebar{\cdot}_1$ is maximised at the same vector. The ratio $\threebar{\cdot}_1/\threebar{\cdot}_2$ is thus both maximised and minimised at the same point, and the ratio is consequently constant on nonzero elements of $\R^4$. We have proved the uniqueness of the Barabanov norm for $\mathsf{X}$ up to scalar multiplication.

\subsection{Failure of strict convexity for the Barabanov norm}  Let $\threebar{\cdot}$ be a Barabanov norm for $\mathsf{X}$. It remains to show that $\threebar{\cdot}$ is not strictly convex. Let $u, v \in \R^2$ be nonzero vectors,
\[u=\begin{pmatrix} u_1 \\ u_2\end{pmatrix},\qquad v=\begin{pmatrix}v_1\\v_2\end{pmatrix},\]
where $|u_2| \leq |u_1|$. We claim that
\begin{equation}\label{eq:bonanza}\threebar{\begin{pmatrix} u_1 \\ u_2\end{pmatrix}\otimes \begin{pmatrix} v_1 \\ v_2\end{pmatrix}} =\threebar{\begin{pmatrix} u_1 \\ 0\end{pmatrix}\otimes \begin{pmatrix} v_1 \\ v_2\end{pmatrix}}.\end{equation}
Since $u_2 \in [-u_1, u_1]$ can be chosen arbitrarily when the other variables $u_1, v_1, v_2$ are fixed, this implies that $\threebar{\cdot}$ is constant on a line segment  and consequently is not strictly convex. 

We approach \eqref{eq:bonanza} by proving the two directions of inequality separately. We begin by showing that the left-hand side of \eqref{eq:bonanza} is greater than or equal to the right-hand side. In this we require only the fact that $\threebar{\cdot}$ is an extremal norm for $\mathsf{X}$. By Theorem \ref{th:consider-the-lilies}\eqref{it:cheezburger} there exist a measurable function $\beta \colon [0,\infty) \to [0,1]$ and an absolutely continuous function $y \colon [0,\infty) \to \R^2$ such that
\[y'(t)=(1-\beta(t))B_0y(t)+\beta(t)B_1y(t)\]
a.e, such that $y(0)=v$, and such that $y$ is periodic with some period $\tau>0$, say. Let $z \colon [0,\infty) \to \R^4$ be the trajectory of $\mathsf{X}$ which solves
\[z'(t) = (1-\beta(t))X_0z(t)+\beta(t)X_1z(t),\qquad z(0)=u \otimes v.\]
Then we have $z(t)=(e^{tA_0}x(0)) \otimes y(t)$ for all $t \geq 0$, and therefore using the periodicity of $y$
\[\lim_{n \to \infty} z(n\tau ) = \lim_{n \to \infty} \left(e^{n\tau A_0}x(0)\right) \otimes y(0) =\begin{pmatrix} u_1 \\ 0\end{pmatrix}\otimes \begin{pmatrix} v_1 \\ v_2\end{pmatrix}.\]
Since $\threebar{\cdot}$ is an extremal norm for $\mathsf{X}$ it follows that
\[\threebar{\begin{pmatrix} u_1 \\ 0\end{pmatrix}\otimes \begin{pmatrix} v_1 \\ v_2\end{pmatrix}} = \lim_{n\to \infty} \threebar{z(n\tau)} \leq \threebar{z(0)} = \threebar{\begin{pmatrix} u_1 \\ u_2\end{pmatrix}\otimes \begin{pmatrix} v_1 \\ v_2\end{pmatrix} }\]
and this gives one direction of inequality in \eqref{eq:bonanza}.

We now prove the other direction of inequality. Using the fact that $\threebar{\cdot}$ is a Barabanov norm for $\mathsf{X}$, let $z \colon [0,\infty) \to \R^4$ be a trajectory of the linear switched system defined by $\mathsf{X}$ such that $z(0)=u \otimes v$ and $\threebar{z(t)}=\threebar{z(0)}$ for all $t \geq 0$, 
and let $\alpha, \beta \colon [0,\infty) \to [0,1]$ be measurable functions such that $0 \leq \alpha + \beta \leq 1$ and such that 
\[z'(t)= (1-\alpha(t) - \beta(t))X_0z(t) + \beta(t)X_1z(t)+\alpha(t) X_2 z(t)\]
for a.e. $t \geq 0$. Let $x, y \colon [0,\infty) \to \R^2$ solve the initial value problems
\[x'(t) = (1-\alpha(t))A_0x(t)+\alpha(t)A_1x(t),\qquad x(0)=u,\]
\[y'(t) =(1-\alpha(t) - \beta(t))B_0y(t) + \beta(t)B_1y(t),\qquad y(0)=v,\]
and observe as before that $z(t)=x(t)\otimes y(t)$ for all $t \geq 0$, that $x$ is a trajectory of the linear switching system defined by $\A$, and that $y$ is a trajectory of the linear switching system defined by $\B_0$. Let $\threebar{\cdot}_\A$ and $\threebar{\cdot}_\B$ be the norms on $\R^2$ which were constructed in Theorems \ref{th:a-pepper} and \ref{th:consider-the-lilies} respectively. Since $\threebar{z(t)}=\threebar{z(0)}> 0$ for all $t \geq 0$, and since $y(t)$ is bounded with respect to $t$, it is impossible that $\lim_{t \to \infty}x(t)=0$ and we deduce using Theorem \ref{th:a-pepper}\eqref{it:three} that necessarily $\int_0^\infty \alpha(t)dt<\infty$. By Theorem \ref{th:consider-the-lilies}\eqref{it:orly} there consequently exists a sequence of positive real numbers $(t_n)$ diverging to infinity such that $\|y(t_n)\|^{-1}y(t_n) =\|y(0)\|^{-1}y(0)$ for every $n \geq 1$, so we have $y(t_n) =(\threebar{y(t_n)}_{\B}/\threebar{y(0)}_\B)y(0)$ for all $n \geq 1$. Since $y$ is a trajectory of the linear switching system defined by $\B_0$, and since $\threebar{\cdot}_\B$ is an extremal norm for $\mathsf{B}_0$, the sequence of values $\threebar{y(t_n)}_\B$ is non-increasing. It follows that
\[\lim_{n \to \infty} y(t_n) = \left(\inf_{n \geq 1} \frac{\threebar{y(t_n)}_\B}{\threebar{y(0)}_\B}\right) \cdot y(0) = \gamma_1 \cdot y(0) = \gamma_1 \cdot\begin{pmatrix}v_1\\v_2\end{pmatrix},\]
say, where $0\leq \gamma_1\leq 1$. On the other hand since $x(t)$ converges as $t \to \infty$ to a vector on the horizontal axis in $\R^2$, we have
\[\lim_{t\to \infty} x(t)=\gamma_2 \cdot \begin{pmatrix}u_1 \\ 0\end{pmatrix},\]
say, where 
\[|\gamma_2|\cdot\threebar{ \begin{pmatrix}u_1 \\ 0\end{pmatrix}}_\A=\lim_{t \to \infty}\threebar{x(t)}_\A \leq \threebar{x(0)}_\A =  \threebar{\begin{pmatrix}u_1 \\ u_2\end{pmatrix}}_\A=  \threebar{\begin{pmatrix}u_1 \\ 0\end{pmatrix}}_\A.\]
Here we have used the fact that $x(t)$ is a trajectory of the linear switching system defined by $\A$, the fact that $\threebar{\cdot}_\A$ is an extremal norm for $\A$, and \eqref{eq:stonks}. Thus we have
\[\lim_{n  \to \infty} z(t_n) = \gamma\cdot \begin{pmatrix} u_1 \\ 0\end{pmatrix} \otimes \begin{pmatrix} v_1 \\ v_2\end{pmatrix} \]
with $|\gamma|=|\gamma_1\gamma_2|\leq 1$, so that
\[\threebar{\begin{pmatrix} u_1 \\ u_2\end{pmatrix} \otimes \begin{pmatrix} v_1 \\ v_2\end{pmatrix}}  =\threebar{z(0)}=\lim_{n \to \infty} \threebar{z(t_n)}\leq \threebar{\begin{pmatrix} u_1 \\ 0\end{pmatrix} \otimes \begin{pmatrix} v_1 \\ v_2\end{pmatrix}} .\]
We have proved \eqref{eq:bonanza}, and the proof of the theorem is complete.

\bibliographystyle{acm}
\bibliography{pork}
\end{document}